\documentclass[11pt, reqno]{amsart}
\usepackage[margin=1in]{geometry}
\usepackage{graphicx}
\usepackage{relsize}
\usepackage{amsmath}
\usepackage{xcolor}
\usepackage{braket}
\newcommand\numberthis{\addtocounter{equation}{1}\tag{\theequation}}

\numberwithin{equation}{section}
\theoremstyle{plain}
\newtheorem{thm}{Theorem}[section]
\newtheorem*{thm*}{Theorem}

\newtheorem{lem}[thm]{Lemma}

\newtheorem*{cor*}{Corollary}
\theoremstyle{definition}

\newtheorem{conj}{Conjecture}[section]

\theoremstyle{remark}
\newtheorem*{rem}{Remark}

\newcommand{\BigO}[1]{\ensuremath{\operatorname{O}\left(#1\right)}}

\makeatletter
\def\pmod#1{\allowbreak\mkern10mu({\operator@font mod}\,\,#1)} 
\makeatother

\begin{document}
\subjclass[2010]{11M06, 11R18, 11N37}
\keywords{Ramanujan Sums, Number Fields, Dedekind Zeta function, Lindel\"of Hypothesis}

\title{Moments of Averages of Ramanujan Sums over Number Fields}

\author{Sneha Chaubey}
\email{sneha@iiitd.ac.in}

\author{Shivani Goel}
\email{shivanig@iiitd.ac.in}

\address[]{Department of Mathematics, IIIT Delhi, New Delhi 110020}

\begin{abstract}
Assuming the generalized Lindel\"of hypothesis, we provide asymptotic formulas for the mean values of the first and second moments of Ramanujan sums over any number field. Additionally, unconditionally, we estimate the second moment of Ramanujan sums over cyclotomic number fields.
\end{abstract}
\maketitle
\section{Introduction and Main results}
For positive integers $q$ and $n$,  Ramanujan \cite{ramanujan1918certain} studied a function defined as follows:
\[c_q(n):=\sum_{\substack{1\le j\le n\\(j,q)=1}}e\left(\frac{ nj}{q}\right)=\sum_{\substack{d|n\\d|q}}d\mu\left(\frac{q}{d}\right).\numberthis\label{rsum}\]
Ramanujan first encountered these sums while exploring trigonometric series representations of normalized arithmetic functions of the form $\sum_{q}a_qc_q(n)$, now known as Ramanujan expansions. These sums \eqref{rsum} have since been recognized as Ramanujan sums. Subsequently, Carmichael \cite{Carmi} established that these sums also possess an orthogonality property.
   Ramanujan and Carmichael's work set the stage for a general theory of Ramanujan sums and Ramanujan expansions.
   Ramanujan sums exhibit deep connections in number theory and arithmetic, including their role in proving Vinogradov's theorem \cite[Chapter 8]{nathanson1996additive}, Waring-type formulas \cite{konvalina1996generalization}, the distribution of rational numbers in short intervals \cite{jutila2007distribution}, equipartition modulo odd integers \cite{balandraud2007application}, the large sieve inequality \cite{ramare2007eigenvalues}, and various other branches of mathematics. For more recent developments in the direction of Ramanujan expansions, we refer \cite{chaubey2023hardy, HD, AD, LR, schwarz1988ramanujan, SS, AW, EW}. 
   


  The main objective of this article is to investigate the moments associated with Ramanujan sums defined over number fields.
  For number field $\mathbb{Q}$, this problem has been studied before in \cite{MR2869206} and \cite{MR3600410}. 
For a number field $\mathbb{K}$, denote $\mathcal{O}_{\mathbb{K}}$ its ring of integers. If $\mathcal{I}$ and $\mathcal{J}$ are non-zero integral ideals within $\mathcal{O}_{\mathbb{K}}$, we define Ramanujan sums over number fields as follows:
        \[C_{\mathcal{J}}({\mathcal{I}}):=\sum_{\substack{\mathcal{I}_1|\mathcal{J}\\\mathcal{I}_1|\mathcal{I}}}\mathcal{N}(\mathcal{I}_1)\mu\left(\frac{\mathcal{J}}{\mathcal{I}_1}\right).\numberthis\label{one}\]
         Here, $\mathcal{N}(\mathcal{I})$ is the norm of $\mathcal{I}$ and $\mu(\mathcal{I})$ is the generalized M\"obius function defined as\[\mu(\mathcal{I}):=\left\{\begin{array}{ll}
       (-1)^r  & \mbox{if } \mathcal{I}\ \text{is a product of r distinct prime ideals},  \\
       0  &  \mbox{if } \text{there exists a prime ideal}\ \mathcal{P} \ \text{of}\ \mathcal{O}_{\mathbb{K}}\  \text{such that} \ \mathcal{P}^2|\mathcal{I}.
    \end{array}\right.\]
    Nowak \cite{Nowak12} showed that if $\mathbb{K}$ is a fixed quadratic number field, and $y>x^{\delta}$, where $\delta>\frac{1973}{820}=2.40609\dots$, then 
\[\sum_{0<\mathcal{N}(\mathcal{I})\le y}\sum_{0<\mathcal{N}(\mathcal{J})\le x}C_{\mathcal{J}}({\mathcal{I}})\sim \rho_{\mathbb{K}} y,\numberthis\label{nowak}\]
\[\rho_{\mathbb{K}}=\lim_{t\rightarrow\infty}\frac{1}{t}\#\{\text{integral ideals }  \mathcal{I} \text{ in } O_{\mathbb{K}}: 0< N(\mathcal{I})\le t\}.\]
He also has a result with an explicit error term. In \cite{nowak13}, the investigation focused on the first moment concerning Gaussian integers. Subsequently, Nowak's result was improved in \cite{zhai2021average} for $\delta>2.3235\dots$. For cubic number number fields, the first moment is derived in \cite{ma2021average}. In this study, the authors obtained an asymptotic formula as in \eqref{nowak}, for $y>x^{11/4}$. Moreover, when considering the first moment of Ramanujan sums over general number fields, the only known result is attributed to Fujisawa \cite{MR3332952}. He proved that  for  any $\delta>\frac{2-\alpha}{1-\alpha}$, where $\alpha\in [0,1)$, with condition $ y\gg x^{\delta}$, and $y\to\infty$, \[\sum_{0<\mathcal{N}(\mathcal{I})\le y}\sum_{0<\mathcal{N}(\mathcal{J})\le x}C_{\mathcal{J}}({\mathcal{I}})= \rho_{\mathbb{K}}y+o(y).\numberthis\label{fujisawa}\]
Landau \cite{landau1949einfuhrung} provided an estimate for the value of $\alpha$, which is determined to be $(n-1)/(n+1)$, with $n$ representing the degree of $\mathbb{K}$ over $\mathbb{Q}$.

It is evident that Fujisawa's finding exhibits limitations and does not hold for small values of $y$. Consequently, we estimate the first moment in cases where $y>x^2$, under the Generalized Lindel\"of hypothesis (GLH). Our result also enhances Nowak and Zhai's outcomes for quadratic and cubic number fields, respectively. We proceed by first stating the GLH.
\begin{conj}[Generalized Lindel\"of hypothesis]\label{lindelofconj}
       Let $\zeta_{\mathbb{K}}(s)$ be the Dedekind zeta function over the number field $\mathbb{K}$, then for any $\epsilon>0$, we have 
       \[\left|\zeta_{\mathbb{K}}\left(\frac{1}{2}+it\right)\right|\ll |t|^{\epsilon}.\]
    \end{conj}
  Following this, we proceed to present our results.
    \begin{thm}\label{theorem1}
Let $\mathbb{K}$ be a number field, then under GLH and for any  $\epsilon>0$ if $y>x^{2}$, we have
\[ \sum_{0<\mathcal{N}(\mathcal{I})\le y}\sum_{0<\mathcal{N}(\mathcal{J})\le x}C_{\mathcal{J}}({\mathcal{I}})=\rho_{\mathbb{K}} y+\BigO{xy^{1/2+\epsilon}\log x}.\]
\end{thm}
In a recent study by the authors \cite{CG}, the distribution of Ramanujan sums through the second moment was investigated, specifically focusing on quadratic and cubic number fields. To be more precise, we showed that for any arbitrary small $\epsilon>0$ and $y\le x^{k_1}$,
\begin{align*}
    \sum_{0<\mathcal{N}(\mathcal{I})\le y} \left(\sum_{0<\mathcal{N}(\mathcal{J})\le x}C_{\mathcal{J}}({\mathcal{I}})\right)^2&=\frac{\rho_{\mathbb{K}}^2\zeta_{\mathbb{K}}(0)}{4\zeta_{\mathbb{K}}(2)^2 }x^4+E(x,y),
\end{align*}
 for $x^{k_1}\le y <x^{k_2}$
   \begin{align*}
        \sum_{0<\mathcal{N}(\mathcal{I})\le y} \left(\sum_{0<\mathcal{N}(\mathcal{J})\le x}C_{\mathcal{J}}({\mathcal{I}})\right)^2&= \frac{\rho_{\mathbb{K}}^2}{2\zeta_{\mathbb{K}}(2)}yx^2+\frac{\rho_{\mathbb{K}}^2\zeta_{\mathbb{K}}(0)}{4\zeta_{\mathbb{K}}(2)^2 }x^4+E(x,y),
   \end{align*}
   and for $y\ge x^{k_2} $
    \begin{align*}
        \sum_{0<\mathcal{N}(\mathcal{I})\le y} \left(\sum_{0<\mathcal{N}(\mathcal{J})\le x}C_{\mathcal{J}}({\mathcal{I}})\right)^2&= \frac{\rho_{\mathbb{K}}^2}{2\zeta_{\mathbb{K}}(2)}yx^2+E(x,y).
   \end{align*}
   In the case of a quadratic number field, we have specific values for the parameters: $k_1=11/9-\epsilon$ and $k_2=36/17-\epsilon$. Additionally, the error term $E(x,y)$ satisfies the bound $E(x,y)\ll x^{47/18+\epsilon}y^{1/2}\log^{12}x+x^2y^{17/18}\log^{24}x+yx^{{5}/{3}+\epsilon}\log^7x$. For a cubic number field, the parameters are different, with $k_1=237/196-\epsilon$ and $k_2=98/45-\epsilon$, and the corresponding error term is bounded as follows: $E(x,y)\ll x^{1021/392+\epsilon}y^{1/2}\log ^{20}x+x^2y^{45/49+\epsilon}\log^{3}x+yx^{25/14+\epsilon}\log^{10}x$. Furthermore, in the same paper, the authors also explored moments for a broader category known as Pr\"ufer domains. In here, we extend the investigation and prove results for any number field. 
   
\begin{thm}\label{theorem2}
For a number field $\mathbb{K}$, and $\epsilon>0$, under GLH, we have:

For  $y<x^{5/2}$, 
  \begin{align*}
        \sum_{0<\mathcal{N}(\mathcal{I})\le y} \left(\sum_{0<\mathcal{N}(\mathcal{J})\le x}C_{\mathcal{J}}({\mathcal{I}})\right)^2&= \frac{\rho_{\mathbb{K}}^2}{\zeta_{\mathbb{K}}(2)}yx^2+\frac{\rho_{\mathbb{K}}^2\zeta_{\mathbb{K}}(0)}{4\zeta_{\mathbb{K}}(2)^2 }x^4+\BigO{yx^{3/2+\epsilon}},
   \end{align*}
   and for $y\ge x^{5/2}$, we have 
   \begin{align*}
        \sum_{0<\mathcal{N}(\mathcal{I})\le y} \left(\sum_{0<\mathcal{N}(\mathcal{J})\le x}C_{\mathcal{J}}({\mathcal{I}})\right)^2&= \frac{\rho_{\mathbb{K}}^2}{\zeta_{\mathbb{K}}(2)}yx^2+\BigO{yx^{3/2+\epsilon}}.
   \end{align*}
\end{thm}
Furthermore, we provide an estimation for cyclotomic number fields without requiring the assumption of the Generalized Lindel\"of hypothesis.
\begin{thm}\label{theorem3}
Let $\mathbb{K} =\mathbb{Q}(\zeta_m)$ be a cyclotomic number field, then  for  $y<x^{2}$
\begin{align*}
   \sum_{0<\mathcal{N}(\mathcal{I})\le y} \left(\sum_{0<\mathcal{N}(\mathcal{J})\le x}C_{\mathcal{J}}({\mathcal{I}})\right)^2&=\frac{\rho_{\mathbb{K}}^2\zeta_{\mathbb{K}}(0)}{4\zeta_{\mathbb{K}}(2)^2 }x^4+\BigO{x^{2-1/4\phi^(m)}y\log^{4\phi(m)+1}x+x^2y^{5/6}\log^{4\phi(m)}x}, 
\end{align*}
and for $y>x^{2}$
  \begin{align*}
        \sum_{0<\mathcal{N}(\mathcal{I})\le y} \left(\sum_{0<\mathcal{N}(\mathcal{J})\le x}C_{\mathcal{J}}({\mathcal{I}})\right)^2&= \frac{\rho_{\mathbb{K}}^2}{2\zeta_{\mathbb{K}}(2)}yx^2+\BigO{x^2y^{5/6}\log^{4\phi(m)}x +x^{5/2-1/2\phi(m)}y^{1/2}\log^{4\phi(m)}x}\\&+\BigO{x^{2-1/4\phi(m)}y\log^{4\phi(m)+1}x}.
   \end{align*}
\end{thm}

\begin{rem}
    Note that using the current best-known bounds of the Dedekind zeta function for any arbitrary cyclotomic number field $\mathbb{Q}(\zeta_m)$, it is not possible to obtain an asymptotic formula like \eqref{nowak} for the first moment.
\end{rem}
\subsection{Organization}
This article is organized as follows. Section \ref{sec2} contains preliminary results required to prove Theorems \ref{theorem1}, \ref{theorem2}, and \ref{theorem3}. In Section \ref{sec3}, we prove key results involving the average of the divisor function and the product of divisor functions over number fields. Section \ref{sec4} contains proof of Theorem \ref{theorem1}.  Section \ref{sec5} contains  proof of Theorem \ref{theorem2}. Finally, in Section \ref{sec6}, we give a brief sketch of the proof of Theorem \ref{theorem3}.  
\subsection{Notations} 
\begin{enumerate}
    \item  We use $\mathbb{Z}$, $\mathbb{Q}$, and $\mathbb{K}$ to denote the set of integers, a set of rational numbers, and a number field of degree $m$, respectively.
    \item We use $\phi$ to denote the Euler totient function, $\mu$  to denote the Mobius function. 
    \item  $\zeta(s)$ is the Riemann zeta function, $\zeta_{\mathbb{K}}(s)$ denotes the Dedekind zeta function corresponding to a number field $\mathbb{K}$, and $L(s,\chi)$ is the Dirichlet L-function. Here $\chi$ is a Dirichlet character.  
          \item $\zeta _{m}$ denotes a  primitive $m$-th root of unity and $\chi_0$ denotes the principal Dirichlet character.
\item  We use the Vinogradov $\ll$ asymptotic notation, and the big oh O(·) and o(·) asymptotic notation. 
\end{enumerate}
\subsection{Acknowledgements}
The authors are grateful to the support from the University Grants Commission, the Department of Higher Education, Government of India [DEC18-434199 to S.G.]; and the Science and Engineering Research Board, Department of Science and Technology, Government of India [SB/S2/RJN-053/2018 to S.C.]. The authors are also grateful to Biswajyoti Saha for suggesting the problem for cyclotomic number fields and valuable discussions on this. 
 
    \section{Preliminaries} \label{sec2}
In this section, we state the results required to prove our main Theorems. 
 Initially, we outline the results related to the Dirichlet series of Ramanujan sums and generalized divisor functions over number fields. These results have been previously demonstrated in  \cite{CG}.
\begin{lem}\cite[Lemma 2.1]{CG}\label{lemmaone}
For a number field $\mathbb{K}$ and for $\Re(s)>1$, one has\[\sum_{\mathcal{J}\subseteq \mathcal{O}_{\mathbb{K}}}\frac{C_{\mathcal{J}}({\mathcal{I}})}{\mathcal{N}(\mathcal{J})^s}=\frac{\sigma_{\mathbb{K},(1-s)}(\mathcal{I})}{\zeta_{\mathbb{K}}(s)},\numberthis\label{two}\]
where $\sigma_{\mathbb{K},(1-s)}(\mathcal{I})=\sum_{\mathcal{I}_1|\mathcal{I}}\mathcal{N}(\mathcal{I}_1)^{1-s}.$
\end{lem}

\begin{lem}\cite[Lemma 2.2]{CG}\label{lemmatwo}
For $z\in \mathbb{C}$, and for a number field $\mathbb{K}$,
\[\sum_{\mathcal{I}\subseteq \mathcal{O}_{\mathbb{K}}}\frac{\sigma_{\mathbb{K},z}(\mathcal{I})}{\mathcal{N}(\mathcal{I})^s}=\zeta_{\mathbb{K}}(s)\zeta_{\mathbb{K}}(s-z),\numberthis\label{three}\]
for $\Re(s)>\max(1+\Re(z),1).$
\end{lem}

\begin{lem}\cite[Lemma 2.3]{CG}\label{lemmafive}
For $\Re(s)>\max(1,1+\Re(z_1),1+\Re(z_2),1+\Re(z_1+z_2))$, we have\[\sum_{\mathcal{I}\subseteq \mathcal{O}_{\mathbb{K}}}\frac{\sigma_{\mathbb{K},z_1}(\mathcal{I})\sigma_{\mathbb{K},z_2}(\mathcal{I})}{\mathcal{N}(\mathcal{I})^s}=\frac{\zeta_{\mathbb{K}}(s)\zeta_{\mathbb{K}}(s-z_1)\zeta_{\mathbb{K}}(s-z_2)\zeta_{\mathbb{K}}(s-z_1-z_2)}{\zeta_{\mathbb{K}}(2s-z_1-z_2)}.\numberthis\label{sigm,adirichlet}\]
\end{lem}

Moving forward, we reference two lemmas that will prove valuable in the subsequent section. The first lemma is a Brun-Titchmarsh theorem established by Shiu \cite{shiu1980brun}, which we will utilize to evaluate the partial sums\[\sum_{\substack{\mathcal{I}\subseteq \mathcal{O}_{\mathbb{K}}\\\mathcal{N}(\mathcal{I})=n}}\sigma_{\mathbb{K},z}(\mathcal{I}),\] and 
\[\sum_{\substack{\mathcal{I}\subseteq \mathcal{O}_{\mathbb{K}}\\\mathcal{N}(\mathcal{I})=n}}\sigma_{\mathbb{K},z_1}(\mathcal{I})\sigma_{\mathbb{K},z_2}(\mathcal{I}).\]
In \cite{shiu1980brun}, the author derives the theorem for a larger class $M$ of arithmetic function $f$ which are non-negative and multiplicative, and which satisfy the following conditions:
\begin{enumerate}
    \item For a prime $p$, and integer $l\ge 1$, there exists a positive constant $C_1$ such that \[f(p^l)\le C_1^l,\]
    \item For every $\epsilon>0$, and for $n\ge 1$, there exists a positive constant $C_2=C_2(\epsilon)$ such that \[f(n)\le C_2n^{\epsilon}.\]
\end{enumerate}
\begin{lem}\cite[Theorem 1]{shiu1980brun} \label{averagelemma}
Let $f\in M$, $0<\alpha, \beta<1/2$, and $a,k$ be integers. If $0<a<k$, and  $(a,k)=1$, then as $x\to \infty$
\[\sum_{\substack{x-y<n\le x\\n\equiv a\mod{q}}} f(n)\ll \frac{y}{\phi(q)\log x}\exp{\left(\sum_{\substack{p\le x\\p\not|q}}\frac{f(p)}{p}\right)},\] uniformly in $a,q$, and $y$ provided that $q\le y^{1-\alpha}$, and  $x^{\beta}<y\le x$.
\end{lem}
The second lemma is a Parron-type formula for a sequence of complex numbers.
\begin{lem}\cite[Lemma 2.8]{MR3600410}\label{parronlemma}
Let $0<\lambda_1<\lambda_2<\cdots<\lambda_n\to \infty$ be any sequence of real numbers, and let $\{a_n\}$ be a sequence of complex numbers. Let the Dirichlet series $g(s):=\sum_{n=1}^{\infty}a_n\lambda_{n}^{-s}$ be absolutely convergent for $\sigma_a$. If $\sigma_0>\max(0,\sigma_a)$ and $x>0$, then \[\sum_{\lambda_n\le x}a_n=\frac{1}{2\pi i}\int_{\sigma_0-iT}^{\sigma_0+iT}g(s)\frac{x^s}{s}ds+R,\] where \[R\ll \sum_{\substack{x/2<\lambda_n<2x\\n\ne x}}|a_n|\min\left(1,\frac{x}{T|x-\lambda_n|}\right)+\frac{4^{\sigma_0}+x^{\sigma_0}}{T}\sum_{n=1}^{\infty}\frac{|a_n|}{\lambda_n^{\sigma_0}}.\] 
\end{lem}

 Next, recall the Phragm\'en-Lindel\"of principle given by
    \begin{thm}\cite[Theorem 5.53]{Iwaniec}\label{paralinde}
        Let $f(\sigma+it)$ be analytic in the strip $a\le\sigma\le b$ with $f(\sigma+it)\ll\exp({\epsilon |t|})$. If $|f(a + it)|\ll |t|^{c_1}$ and $|f(b + it)|\ll |t|^{c_2}$, then
\[|f(\sigma + it)|\ll |t|^{c(\sigma)},\]
uniformly in $a\le\sigma\le b$, where $c(\sigma)$ is linear in $\sigma$ with $c(a)=c_1$ and $c(b)=c_2$.
    \end{thm}
    Using Generalized Lindel\"of Hypothesis and the Phragm\'en-Lindel\"of principle, we have the following upper bounds of a Dedekind zeta function associated with a number field $\mathbb{K}$ of degree $m.$
    \[|\zeta_{\mathbb{K}}(\sigma+it)|\ll\left\{\begin{array}{lll}
      |t|^{m(1/2-\sigma)+\epsilon}  & \mbox{if } 0\le \sigma \le 1/2,  \\
       |t|^{\epsilon}  & \mbox{if } 1/2\le \sigma \le 1\\
       \log |t| & \mbox{if } 1\le \sigma \le 2.
    \end{array}\right.\numberthis\label{lindelofbound}\]

    \section{Averages of generalized divisor function}\label{sec3}
    In this section, we prove Lemmas for the average value of the generalized divisor function and the product of two generalized divisor functions. These Lemmas are the key estimates for main Theorems.
\begin{lem}\label{lemmathree}
Let $\mathbb{K}$ be a number field, $-1/2<\Re(z)\le0$, and $|\Im(z)|\le x$, then  under GLH and for arbitrary small $\epsilon>0$, we have \[\sum_{0<\mathcal{N}(\mathcal{I})\le n}\sigma_{\mathbb{K},z}(\mathcal{I})=\rho_{\mathbb{K}} \zeta_{\mathbb{K}}(1-z) x+\rho_{\mathbb{K}} \zeta_{\mathbb{K}}(1+z)\frac{x^{1+z}}{1+z}+ \BigO{x^{1/2+\epsilon}}.\]
\end{lem}
\begin{proof}
We write\[\sum_{\mathcal{I}\subseteq \mathcal{O}_{\mathbb{K}}}\frac{\sigma_{\mathbb{K},z}(\mathcal{I})}{\mathcal{N}(\mathcal{I})^s}=\sum_{n=1}^{\infty}\frac{1}{n^s}\sum_{\mathcal{N}(\mathcal{I})=n}\sigma_{\mathbb{K},z}(\mathcal{I})=\sum_{n=1}^{\infty}\frac{A(n,z)}{n^s},\]
where $A(n,z)=\sum_{\mathcal{N}(\mathcal{I})=n}\sigma_{\mathbb{K},z}(\mathcal{I}).$ Consider $c=1+1/\log x$ and $z=a+ib$ with $-1/2<a<0$. Therefore, from Lemma \ref{parronlemma}, we have\[\sum_{n\le x}{A(n,z)}=\frac{1}{2\pi i}\int_{c-iT}^{c+iT}\zeta_{\mathbb{K}}(s)\zeta_{\mathbb{K}}(s-z)\frac{x^s}{s}ds+R(x,z),\numberthis\label{parron}\] where \[R(x,z)\ll\sum_{x/2<n<2x}|A(n,z)|\min\left(1,\frac{x}{T|x-n|}\right)+\frac{x^c}{T}\sum_{n=1}^{\infty}\frac{|A(n,z)|}{n^c}.\numberthis\label{four}\]
Next, from \eqref{three}, we conclude that
\begin{align*}
    |A(n,z)|&\le \sum_{d|n}\left(\sum_{d_1\cdots d_{m}|d}1\sum_{l_1\cdots l_{m}|n/d}l_1^a\cdots l_{m}^a\right)\\
    &\le \sum_{d|n}\left(\sum_{d_1\cdots d_{m}|d}1\sum_{l_1\cdots l_{m}|n/d}1\right)\numberthis\label{five}.
\end{align*}
For $\alpha$ is a positive real number, we choose $T=x^{\alpha}$.  Now, if $0<|x-n|<x^{1-\alpha}$, then \eqref{five} gives 
\begin{align*}
    \sum_{0<|x-n|<x^{1-\alpha}}|A(n,z)|\min\left(1,\frac{x}{T|x-n|}\right)&\ll \sum_{0<|x-n|<x^{1-\alpha}}\sum_{d|n}\left(\sum_{d_1\cdots d_{m}|d}1\sum_{l_1\cdots l_{m}|n/d}1\right).\numberthis\label{six}
\end{align*}
Moreover, using Lemma \ref{averagelemma} and \eqref{six}, we have \[\sum_{0<|x-n|<x^{1-\alpha}}|A(n,z)|\min\left(1,\frac{x}{T|x-n|}\right)\ll x^{1-\alpha}\log^{2m-1}x.\]
Next, if $x+x^{1-\alpha}<n<2x$, then \begin{align*}
   & \sum_{x+x^{1-\alpha}<n<2x}|A(n,z)|\min\left(1,\frac{x}{T|x-n|}\right)\\&\ll\frac{x}{T}\sum_{x+x^{1-\alpha}<n<2x}\frac{\sum_{d|n}\left(\sum_{d_1\cdots d_{m}|d}1\sum_{l_1\cdots l_{m}|n/d}1\right)}{n-x}\\
    & \ll \frac{x}{T}\sum_{l\ll \log x}\frac{1}{U}\sum_{\substack{U<n-x<2U\\U=2^lx^{1-\alpha}}}\frac{\sum_{d|n}\left(\sum_{d_1\cdots d_{m}|d}1\sum_{l_1\cdots l_{m}|n/d}1\right)}{n-x}\\&\ll \frac{x}{T}\log^{2m}x.\numberthis\label{seven}
\end{align*}
For $x/2<n<x-x^{1-\alpha}$, we get the same bound. From Lemma \ref{lemmatwo}, we have \[\frac{x^c}{T}\sum_{n=1}^{\infty}\frac{|A(n,z)|}{n^c}\ll \frac{x}{T}\log^{2m}x\numberthis\label{eight}.\]
Thus, from \eqref{six}, \eqref{seven}, and \eqref{eight}, we obtain
\[R(x,z)\ll x^{1-\alpha}\log^{2m}x.\numberthis\label{error}\]
Next, we evaluate the integral in \eqref{parron} by shifting the line integration $c-iT$ to $c+iT$ into the rectangular contour consists the line segments $I_1:c-iT$ to $1/2-iT$, $I_2:1/2-iT$ to $1/2+iT$, $I_3:1/2+iT$ to $c+iT$, and $I_4:c+iT$ to $c-iT$. By Cauchy residue theorem 
\[\frac{1}{2\pi i}\int_{c-iT}^{c+iT}\zeta_{\mathbb{K}}(s)\zeta_{\mathbb{K}}(s-z)\frac{x^s}{s}ds=\rho_{\mathbb{K}} \zeta_{\mathbb{K}}(1-z)x+\rho_{\mathbb{K}} \zeta_{\mathbb{K}}(1+z)\frac{x^{1+z}}{1+z}+\sum_{i=1}^{3}J_i,\numberthis\label{integ1}\]
where $\rho_{\mathbb{K}}$ is the residue of $\zeta_\mathbb{K}(s)$ at the pole $s=1$. In the right side of \eqref{integ1}, the first and the second terms are the residues at the poles  $1$ and $1+z$, respectively, and the last term is the sum of integration along the line segments $I_i$.

 From  \eqref{lindelofbound}, if $|b|\le T$ then we have  
\begin{align*}
    &|J_1|,|J_3|\ll \int_{1/2}^{c}\frac{|\zeta_{\mathbb{K}}(s)||\zeta_{\mathbb{K}}(s-z)|x^{\sigma}}{T}d\sigma\\
    &\ll\frac{x^{1/2}\log x}{T^{1-\epsilon}}.\numberthis\label{integration1}
\end{align*}

The line integral $J_2$ along the line segment $I_2$ is given by
\begin{align*}
    |J_2|&\ll\int_{-T}^{T}\frac{|\zeta_{\mathbb{K}}(1/2+it)||\zeta_{\mathbb{K}}(1/2-a+it-ib)||x^{1/2+it}|}{|1/2+it|}dt\\&\ll x^{1/2} \int_{-T}^{T}t^{-1+\epsilon} dt=x^{1/2}T^{\epsilon}.\numberthis\label{bound2}
\end{align*}
Choose $T=x^2$ and 
collecting all the estimates from \eqref{error}, \eqref{integration1}, and \eqref{bound2} and inserting in \eqref{parron}, provides us the required result.
\end{proof}

\begin{lem}\label{lemmasix}
Let $-1/2\le \Re(z_1)=a_1<0$, $-1/2m^2\le\Re(z_2)=a_2<0$, and $-1/2<\Re(z_1+z_2)=a_1+a_2<0$, then under GLH and for $\epsilon>0$ we have \[\sum_{\substack{0<\mathcal{N}(\mathcal{I})\le x\\\mathcal{I}\subseteq \mathcal{O}_{\mathbb{K}}}}\sigma_{\mathbb{K},z_1}(\mathcal{I})\sigma_{\mathbb{K},z_2}(\mathcal{I})=\mathrel{R}_0+\BigO{x^{\frac{1+a_1+a_2-4a_2m}{2}+\epsilon}},\]where \begin{align*}
  \mathrel{R}_0&=  \rho_{\mathbb{K}} \frac{\zeta_{\mathbb{K}}(1-z_1)\zeta_{\mathbb{K}}(1-z_2)\zeta_{\mathbb{K}}(1-z_1-z_2)}{\zeta_{\mathbb{K}}(2-z_1-z_2)}{x}+\rho_{\mathbb{K}}\frac{\zeta_{\mathbb{K}}(1+z_1)\zeta_{\mathbb{K}}(1+z_1-z_2)\zeta_{\mathbb{K}}(1-z_2)}{\zeta_{\mathbb{K}}(2+z_1-z_2)}\frac{x^{1+z_1}}{1+z_1}\\&+\rho_{\mathbb{K}}\frac{\zeta_{\mathbb{K}}(1+z_2)\zeta_{\mathbb{K}}(1+z_2-z_1)\zeta_{\mathbb{K}}(1-z_1)}{\zeta_{\mathbb{K}}(2-z_1+z_2)}\frac{x^{1+z_2}}{1+z_2}\\&+\rho_{\mathbb{K}} \frac{\zeta_{\mathbb{K}}(1+z_1+z_2)\zeta_{\mathbb{K}}(1+z_2)\zeta_{\mathbb{K}}(1+z_1)}{\zeta_{\mathbb{K}}(2+z_1+z_2)}\frac{x^{1+z_1+z_2}}{1+z_1+z_2},\end{align*} 
  and $m$ is the degree of ${\mathbb{K}}.$
\end{lem}
\begin{proof}
From Lemma \ref{lemmafive},  we have \[\sum_{\mathcal{I}\subseteq \mathcal{O}_{\mathbb{K}}}\frac{\sigma_{\mathbb{K},z_1}(\mathcal{I})\sigma_{\mathbb{K},z_2}(\mathcal{I})}{\mathcal{N}(\mathcal{I})^s}=\sum_{n=1}^{\infty}\frac{1}{n^s}\sum_{\substack{\mathcal{I}\subseteq \mathcal{O}_{\mathbb{K}}\\\mathcal{N}(\mathcal{I})=n}}\sigma_{\mathbb{K},z_1}(\mathcal{I})\sigma_{\mathbb{K},z_2}(\mathcal{I})=\sum_{n=1}^{\infty}\frac{A(n,z_1,z_2)}{n^s},\] 
where we define $A(n,z_1,z_2):=\sum_{\substack{\mathcal{I}\subseteq \mathcal{O}_{\mathbb{K}}\\\mathcal{N}(\mathcal{I})=n}}\sigma_{\mathbb{K},z_1}(\mathcal{I})\sigma_{\mathbb{K},z_2}(\mathcal{I})$ and \[f(z_1,z_2,s):=\dfrac{\zeta_{\mathbb{K}}(s)\zeta_{\mathbb{K}}(s-z_1)\zeta_{\mathbb{K}}(s-z_2)\zeta_{\mathbb{K}}(s-z_1-z_2)}{\zeta_{\mathbb{K}}(2s-z_1-z_2)}.\]
Let $z_1=a_1+ib_1$ and $z_2=a_2+ib_2$ be two complex numbers such that $a_1,a_2<0$, and $a_1+a_2>-1$. Consider $\alpha=1+\dfrac{1}{\log x}$. Using Lemma \ref{parronlemma}, we have \begin{align*}
   \sum_{n\le x}A(n,z_1,z_2)=&\frac{1}{2\pi i}\int_{\alpha-iT}^{\alpha+iT}f(z_1,z_2,s)\frac{x^s}{s}ds+R(x;z_1,z_2),\numberthis\label{parron4}\end{align*} where \[R(x;z_1,z_2)\ll \sum_{x/2<n<2x}|A(n,z_1,z_2)|\min\left(1,\frac{x}{T|x-n|}\right)+\frac{x^{\alpha}}{T}\sum_{n=1}^{\infty}\frac{|A(n,z_1,z_2)|}{n^{\alpha}}.\numberthis\label{errordoublesum}\]
 To solve integral in \eqref{parron4}, we shift the line integral into a rectangular contour with vertices $\alpha\pm iT$ and $\lambda\pm iT$ where $\lambda=(1+a_1+a_2)/2$. The integrand has poles at $s_1=1$, $s_2=z_1+1$, $s_3=z_2+1$, and $s_4=z_1+z_2+1$ inside the contour. By the Cauchy residue theorem, we have   
 
 \begin{align*}
     \frac{1}{2\pi i}\int_{\alpha-iT}^{\alpha+iT}f(z_1,z_2,s)\frac{x^s}{s}ds&= \mathrel{R_0}+\sum_{i=1}^{3}J_i,\numberthis\label{finalintegral}
 \end{align*}
 where $ \mathrel{R_0}$ is the sum of residues at poles given by
 \begin{align*}
   \mathrel{R_0}&=\rho_{\mathbb{K}} \frac{\zeta_{\mathbb{K}}(1-z_1)\zeta_{\mathbb{K}}(1-z_2)\zeta_{\mathbb{K}}(1-z_1-z_2)}{\zeta_{\mathbb{K}}(2-z_1-z_2)}{x}\\&+\rho_{\mathbb{K}}\frac{\zeta_{\mathbb{K}}(1+z_1)\zeta_{\mathbb{K}}(1+z_1-z_2)\zeta_{\mathbb{K}}(1-z_2)}{\zeta_{\mathbb{K}}(2+z_1-z_2)}\frac{x^{1+z_1}}{1+z_1}\\&+\rho_{\mathbb{K}}\frac{\zeta_{\mathbb{K}}(1+z_2)\zeta_{\mathbb{K}}(1+z_2-z_1)\zeta_{\mathbb{K}}(1-z_1)}{\zeta_{\mathbb{K}}(2-z_1+z_2)}\frac{x^{1+z_2}}{1+z_2}\\&+\rho_{\mathbb{K}} \frac{\zeta_{\mathbb{K}}(1+z_1+z_2)\zeta_{\mathbb{K}}(1+z_2)\zeta_{\mathbb{K}}(1+z_1)}{\zeta_{\mathbb{K}}(2+z_1+z_2)}\frac{x^{1+z_1+z_2}}{1+z_1+z_2},\numberthis\label{residue}\end{align*}
    and $J_i$'s are the line integrals along the lines $[\alpha+iT,\lambda+iT]$, $[\lambda-iT,\lambda+iT]$ and $[\alpha-iT,\lambda-iT]$,  respectively. 

   Next, from \eqref{sigm,adirichlet}, we have  
   \begin{align*}
      |A(n,z_1,z_2)|&\le \sum_{d|n}\left(\sum_{d_1|d}\left(\sum_{d_{11}\cdots d_{1m}|d_1}1\sum_{l_{11}\cdots l_{1m}|n/d}1\right)\sum_{d_2|n/{d_1}}\left(\sum_{d_{21}\cdots d_{2m}|d}1\sum_{l_{21}\cdots l_{2m}|n/{d_2}}1\right)\right).\numberthis\label{azz}
  \end{align*}
This implies $|A(p,z_1,z_2)|\le 4m.$ Taking $T=x^c$ where $c$ is a fixed real number yields \begin{align*}
       \sum_{|x-n|<x^{1-c}}|A(n,z_1,z_2)|\min\left(1,\frac{x}{T|x-n|}\right)&=\sum_{|x-n|<x^{1-c}}|A(n,z_1,z_2)|.
   \end{align*}
   We intend to use Lemma \ref{averagelemma} to estimate the sum on the right side of the above estimate. From \eqref{azz}, we ensure that $A(n,z_1,z_2)$ satisfies the hypothesis in Lemma \ref{averagelemma} and therefore \[\sum_{|x-n|<x^{1-c}}|A(n,z_1,z_2)|\min\left(1,\frac{x}{T|x-n|}\right)\ll \frac{x}{T}\log^{4m}T.\]
For the interval ${x+x^{1-c}<n<2x}$, we have \begin{align*}
  \sum_{x+x^{1-c}<n<2x}|A(n,z_1,z_2)|\min\left(1,\frac{x}{T|x-n|}\right)&\ll\frac{x}{T}\sum_{x+x^{1-\alpha}<n<2x}\frac{|A(n,z_1,z_2)|}{n-x}  \ll \frac{x}{T}\log^{4m}x.
\end{align*}
Moreover, we have the same bounds for $x/2<T<x-x^{1-c}$. Therefore, from the above estimates, we deduce that
\[R(x;z_1,z_2)\ll\frac{x}{T}\log^{4m}T.\numberthis\label{errordouble4} \]
Next,  from Holder's inequality, we can write
 \begin{align*}
\left(\int_{\lambda}^{\alpha}\int_{T_0/2}^{T_0}f(z_1,z_2,\sigma+it)\frac{x^{\sigma+it}}{\sigma+it}d\sigma dt \right)^4&\ll  \int_{\lambda}^{\alpha}\int_{T_0/2}^{T_0}\frac{|\zeta_{\mathbb{K}}(\sigma+it)|^4x^{\sigma}}{|\zeta_{\mathbb{K}}(2(\sigma+it)-z_1-z_2)(\sigma+it)|}d\sigma dt\\&\times\int_{\lambda}^{\alpha}\int_{T_0/2}^{T_0}\frac{|\zeta_{\mathbb{K}}(\sigma+it-z_1)|^4x^{\sigma}}{|\zeta_{\mathbb{K}}(2(\sigma+it)-z_1-z_2)(\sigma+it)|}d\sigma dt\\ &\times \int_{\lambda}^{\alpha}\int_{T_0/2}^{T_0}\frac{|\zeta_{\mathbb{K}}(\sigma+it-z_2)|^4x^{\sigma}}{|\zeta_{\mathbb{K}}(2(\sigma+it)-z_1-z_2)(\sigma+it)|}d\sigma dt \\& \times \int_{\lambda}^{\alpha}\int_{T_0/2}^{T_0}\frac{|\zeta_{\mathbb{K}}(\sigma+it-z_1-z_2)|^4x^{\sigma}}{|\zeta_{\mathbb{K}}(2(\sigma+it)-z_1-z_2)(\sigma+it)|}d\sigma dt.\numberthis\label{holder}
 \end{align*}
Using the bounds in \eqref{lindelofbound}, we have
 \begin{align*}
&\int_{\lambda}^{\alpha}\int_{T_0/2}^{T_0}\frac{|\zeta_{\mathbb{K}}(\sigma+it)|^4x^{\sigma}}{|\zeta_{\mathbb{K}}(2(\sigma+it)-z_1-z_2)(\sigma+it)|}d\sigma dt\\& \ll \int_{T_0/2}^{T_0} \int_{\lambda}^{1/2}t^{m(2-4\sigma)+\epsilon} \frac{x^{\sigma}}{t}d\sigma dt+ \int_{T_0/2}^{T_0}\int_{1/2}^{\alpha}t^{4\epsilon}\frac{x^{\sigma}}{t}d\sigma dt\\&=\int_{T_0/2}^{T_0}t^{2m-1+\epsilon} \int_{\lambda}^{1/2}\left(\frac{x}{t^{4m}}\right)^{\sigma}d\sigma dt+\int_{T_0/2}^{T_0}t^{4\epsilon-1}\int_{1/2}^{\alpha}\left({x}\right)^{\sigma}d\sigma dt\\&={T_0^{2m-4m\lambda+\epsilon}}x^{\lambda}+xT_0^{\epsilon}.
 \end{align*}
 If $a_1-a_2>0$, then we have 
 \begin{align*}
    & \int_{\lambda}^{\alpha}\int_{T_0/2}^{T_0}\frac{|\zeta_{\mathbb{K}}(\sigma+it-z_1)|^4x^{\sigma}}{|\zeta_{\mathbb{K}}(2(\sigma+it)-z_1-z_2)(\sigma+it)|}d\sigma dt \\&\ll  \int_{T_0/2}^{T_0}\int_{\lambda}^{1/2+a_1}t^{m(2-4\sigma+4a_1)+\epsilon} \frac{x^{\sigma}}{t}d\sigma dt+ \int_{T_0/2}^{T_0}\int_{1/2+a_1}^{\alpha}t^{4\epsilon}\frac{x^{\sigma}}{t}d\sigma dt\\&= \int_{T_0/2}^{T_0}t^{2m-1+4a_1m+\epsilon} \int_{\lambda}^{\frac{1}{2}+a_1}\left(\frac{x}{t^{4m}}\right)^{\sigma}d\sigma dt+\int_{T_0/2}^{T_0}t^{4\epsilon-1}\int_{\frac{1}{2}+a_1}^{\alpha}\left({x}\right)^{\sigma}d\sigma dt\\& \ll T_0^{(2a_1-2a_2)m+\epsilon}x^{\lambda}+xT_0^{\epsilon}.
 \end{align*}
 Similarly, \begin{align*}
     & \int_{\lambda}^{\alpha}\int_{T_0/2}^{T_0}\frac{|\zeta_{\mathbb{K}}(\sigma+it-z_2)|^4x^{\sigma}}{|\zeta_{\mathbb{K}}(2(\sigma+it)-z_1-z_2)(\sigma+it)|}d\sigma dt\\ &\ll \int_{T_0/2}^{T_0}\int_{\lambda}^{1+a_2}t^{m(2-4\sigma+4a_2)+\epsilon} \frac{x^{\sigma}}{t}d\sigma dt+\int_{T_0/2}^{T_0}\int_{1+a_2}^{\alpha}t^{4\epsilon}\frac{x^{\sigma}}{t}d\sigma dt\\
      &\ll T_0^{-2m+\epsilon}x^{1+a_2}+{x}T_0^{\epsilon},
 \end{align*}
 and \begin{align*}
   &\int_{\lambda}^{\alpha}\int_{T_0/2}^{T_0}\frac{|\zeta_{\mathbb{K}}(\sigma+it-z_1-z_2)|^4x^{\sigma}}{|\zeta_{\mathbb{K}}(2(\sigma+it)-z_1-z_2)(\sigma+it)|}d\sigma dt\\&\ll  \int_{T_0/2}^{T_0}\int_{\lambda}^{1+a_1+a_2}t^{m(2-4\sigma+4a_1+4a_2)+\epsilon}\frac{x^{\sigma}}{t}d\sigma dt+\int_{T_0/2}^{T_0}\int_{1+a_1+a_2}^{\alpha}t^{4\epsilon}\frac{x^{\sigma}}{t}d\sigma dt\\&\ll x^{1+a_1+a_2}T_0^{-2m+\epsilon}+{x}T_0^{\epsilon}.
 \end{align*} Collecting all the above results and substituting in \eqref{holder}, we obtain 
 \[\int_{\lambda}^{\alpha}\int_{T_0/2}^{T_0}f(z_1,z_2,\sigma+it)\frac{x^{\sigma+it}}{\sigma+it}d\sigma dt\ll {x}T_0^{\epsilon}. \] 
 Next, we choose $T$ such that $T_0/2<T<T_0$, which gives 
 \[\int_{\lambda}^{\alpha}f(z_1,z_2,\sigma+iT)\frac{x^{\sigma+iT}}{\sigma+iT}d\sigma\ll \frac{x}{T^{1-\epsilon}}.\]
 Integral along the vertical line $[\lambda-iT,\lambda+iT]$ is given by \begin{align*}
     \int_{\lambda-iT}^{\lambda+iT}f(z_1,z_2,s)\frac{x^{s}}{s}ds&\ll \int_{-T}^{T}t^{-a_2m+\epsilon}\frac{x^{\lambda}}{t}dt\\ &\ll T^{{-a_2m+\epsilon}}x^{\lambda}. \end{align*}
     Putting $T=x^{2}$  in the above estimates, we get the required result.
  \end{proof}
\section{First Moment}\label{sec4}
\begin{proof}[Proof of theorem \ref{theorem1}]
Let  $\beta=1+\frac{1}{\log y},$ then from Lemma \ref{lemmaone} and Lemma \ref{parronlemma}, we have\[\sum_{0<\mathcal{N}(\mathcal{J})\le x}C_{\mathcal{J}}({\mathcal{I}})=\frac{1}{2\pi i}\int_{\beta-iT}^{\beta+iT}\frac{\sigma_{\mathbb{K},(1-s)}(\mathcal{I})}{\zeta_{\mathbb{K}}(s)}\frac{x^s}{s}ds+R_1(x,\mathcal{I}),\numberthis\label{parron2}\] where 
 \[R_1(x,\mathcal{I})\ll \frac{x}{T}\sigma_{\mathbb{K},(-1/\log y)}(\mathcal{I})+\frac{x\log y}{T}\sigma_{\mathbb{K},0}(\mathcal{I}).\numberthis\label{1momenterror}\]
Next, summing the both sides of \eqref{parron2} over the $\mathcal{N}(\mathcal{I})$ and applying Lemma \ref{lemmathree}, we have 
\begin{align*}
    \sum_{0<\mathcal{N}(\mathcal{I})\le y}\sum_{0<\mathcal{N}(\mathcal{J})\le x}C_{\mathcal{J}}({\mathcal{I}})&=\frac{\rho_{\mathbb{K}} y}{2\pi i}\int_{\beta-iT}^{\beta+iT}\frac{x^s}{s}ds+\rho_{\mathbb{K}}y^2\int_{\beta-iT}^{\beta+iT}\frac{\zeta_{\mathbb{K}}(2-s)}{\zeta_{\mathbb{K}}(s)}\frac{y^{-s}x^s}{(2-s)s}ds\\&+\BigO{y^{1/2+\epsilon}\int_{\beta-iT}^{\beta+iT}\frac{1}{\zeta_{\mathbb{K}}(s)}\frac{x^s}{s}ds+\frac{xy}{T}}\\
    &=I_1+I_2+I_3+I_4.\numberthis\label{integration2}\end{align*}
We  see that  \[I_1=\rho_{\mathbb{K}} y+\BigO{\frac{yx}{T}},\]
and  \[I_3\ll{xy^{1/2+\epsilon}\log^{m}T}.\]
Also, from \eqref{lindelofbound}, we have
\begin{align*}
  I_2 & \ll {xy}\int_{-T}^{T}{t^{-2}\log t}dt \ll \frac{xy\log T}{T}.
\end{align*}
Collecting all the above estimates and choosing $T=x^2$  gives the required result.
\end{proof} 

\section{Second Moment}\label{sec5}
In this section, we prove Theorem \ref{theorem2} by adapting the method of \cite{CG} and \cite{MR3600410}.

\begin{proof}[Proof of Theorem \ref{theorem2}]
Let $\beta_i=1+\frac{i}{\log y}$ where $i\in \{1,2\}.$ Using Lemma \ref{parronlemma}, we have \[\sum_{0<\mathcal{N}(\mathcal{J})\le x}C_{\mathcal{J}}({\mathcal{I}})=\frac{1}{2\pi i}\int_{\beta_i-iT}^{\beta_i+iT}\frac{\sigma_{\mathbb{K},(1-s)}(\mathcal{I})}{\zeta_{\mathbb{K}}(s)}\frac{x^s}{s}ds+\BigO{\frac{x\log y}{T}\sigma_{\mathbb{K},0}(\mathcal{I})}.\]
Squaring both sides
\[\left(\sum_{0<\mathcal{N}(\mathcal{J})\le x}C_{\mathcal{J}}({\mathcal{I}})\right)^2=\frac{1}{(2\pi i)^2}\int_{\beta_1-iT}^{\beta_1+iT}\int_{\beta_2-iT}^{\beta_2+iT}\frac{\sigma_{\mathbb{K},(1-s_1)}(\mathcal{I})\sigma_{\mathbb{K},(1-s_2)}(\mathcal{I})}{\zeta_{\mathbb{K}}(s_1)\zeta_{\mathbb{K}}(s_2)}\frac{x^{s_1+s_2}}{s_1s_2}ds_1ds_2+R(x,\mathcal{I}),\numberthis\label{squaresum}\]
where \begin{align*}
    R(x,\mathcal{I})&\ll \frac{x\log y}{T}\sigma_{\mathbb{K},0}(\mathcal{I})\int_{\beta_i-iT}^{\beta_i+iT}\frac{\sigma_{\mathbb{K},(1-s)}(\mathcal{I})}{\zeta_{\mathbb{K}}(s)}\frac{x^s}{s}ds+ \frac{x^2\log^2 y}{T^2}(\sigma_{\mathbb{K},0}(\mathcal{I}))^2\\
&\ll \frac{x^2\log y\log^mT}{T}(\sigma_{\mathbb{K},0}(\mathcal{I}))^2.\numberthis\label{error7}
\end{align*}
Next, we substitute \eqref{error7} in \eqref{squaresum}, and then summing the both sides over $\mathcal{N}(\mathcal{I})$, this gives \begin{align*}
   &\sum_{0<\mathcal{N}(\mathcal{I})\le y} \left(\sum_{0<\mathcal{N}(\mathcal{J})\le x}C_{\mathcal{J}}({\mathcal{I}})\right)^2\\&=\frac{1}{(2\pi i)^2}\int_{\beta_1-iT}^{\beta_1+iT}\int_{\beta_2-iT}^{\beta_2+iT}\frac{G(s_1,s_2,y)}{\zeta_{\mathbb{K}}(s_1)\zeta_{\mathbb{K}}(s_2)}\frac{x^{s_1+s_2}}{s_1s_2}ds_1ds_2+\BigO{\frac{x^2\log y\log^mT}{T} \sum_{0<\mathcal{N}(\mathcal{I})\le y}(\sigma_{\mathbb{K},0}(\mathcal{I}))^2}\\
   &=I+\BigO{\frac{x^2\log y\log^mT}{T} \sum_{0<\mathcal{N}(\mathcal{I})\le y}(\sigma_{\mathbb{K},0}(\mathcal{I}))^2},\numberthis\label{secondmoment}
\end{align*}
where $G(s_1,s_2,y)=\sum_{0<\mathcal{N}(\mathcal{I})\le y}\sigma_{\mathbb{K},(1-s_1)}(\mathcal{I}) \sigma_{\mathbb{K},(1-s_2)}(\mathcal{I}).$ We use Lemma \ref{lemmasix}  for $a_1=a_2=0$  \[\sum_{0<\mathcal{N}(\mathcal{I})\le y}(\sigma_{\mathbb{K},0}(\mathcal{I}))^2\ll y.\]
From Lemma \ref{lemmasix}, for $|T|\ll x^{2}$, the first term of right side of \eqref{secondmoment} can be written as 
\begin{align*}
    I=I_1+I_2+I_3+I_4+\BigO{x^2y^{1/2+\epsilon}},\numberthis\label{integral}
\end{align*}
where \begin{align*}
    I_1&=\frac{\rho_{\mathbb{K}} y}{(2\pi i)^2}\int_{\beta_1-iT}^{\beta_1+iT}\int_{\beta_2-iT}^{\beta_2+iT} \frac{\zeta_{\mathbb{K}}(s_1+s_2-1)}{\zeta_{\mathbb{K}}(s_1+s_2)}\frac{x^{s_1+s_2}}{s_1s_2}ds_1ds_2.
\end{align*}
 We shift the line integral in $s_2$-plane to a rectangular contour containing the lines $[\beta_2+iT,1/2+iT]$, $[1/2+iT,1/2-iT]$, $[1/2-iT,\beta_2-iT]$, and $[\beta_2+iT,\beta_2-iT]$. We see that $s_2=2-s_1$ is the pole of the integrand inside the contour, and the residue at the pole is \[\dfrac{\rho_{\mathbb{K}} }{\zeta_{\mathbb{K}}(2)s_1(2-s_1)}x^2.\] If $L_{1,1}$, $L_{1,2}$, and $L_{1,3}$ are the integrals along the lines   $[3/2-\beta_1+iT,3/2-\beta_1-iT]$, $[\beta_2+iT,\beta_2-iT]$, and $[3/2-\beta_1-iT,\beta_2-iT]$, respectively, then 
 \begin{align*}
     |L_{1,1}|,|L_{1,3}| &\ll xy\int_{-T}^{T}\left(\int_{1/2}^{\beta_2}T^{\epsilon}\frac{x^{\sigma}}{T}d\sigma\right)\frac{1}{1+|t|}dt
    \ll \frac{x^2y}{T^{1-\epsilon}},\numberthis\label{l1integral}
 \end{align*}
 and the integral along $L_{1,2}$ is given by \begin{align*}
     |L_{1,2}|&\ll yx^{3/2}\int_{-T}^{T}\int_{-T}^{T}\frac{|\zeta_{\mathbb{K}}(\beta_1-1/2+i(t_1+t_2))|}{|\zeta_{\mathbb{K}}(\beta_1+1/2+i(t_1+t_2))|}\frac{1}{(1+|t_1|)(1+|t_2|)}dt_1dt_2\\
     &\ll yx^{3/2}\int_{-2T}^{2T}\frac{|\zeta_{\mathbb{K}}(\beta_1-1/2+it)|}{|\zeta_{\mathbb{K}}(\beta_1+1/2+it)|} \int_{-T}^{T}\frac{1}{(1+|t_1|)(1+|t-t_1|)}dt_1dt\\
     & \ll yx^{3/2}\log T \int_{-2T}^{2T}\frac{t^{\epsilon}}{(1+|t|)}dt\ll yx^{3/2}T^{\epsilon} . \numberthis\label{l2integral}
 \end{align*}
 Therefore, the integral $I_1$ is 
 \begin{align*}
     I_1&= \frac{\rho_{\mathbb{K}}^2yx^2}{\zeta_{\mathbb{K}}(2)}\frac{1}{2\pi i}\int_{\beta_1-iT}^{\beta_1+iT}\frac{1}{s
     _1(2-s_1)}ds_1+\BigO{\frac{x^2y}{T^{1-\epsilon}}+ yx^{3/2}T^{\epsilon}}\\
     &= \frac{\rho_{\mathbb{K}}^2}{2\zeta_{\mathbb{K}}(2)}yx^2+\BigO{\frac{x^2y}{T^{1-\epsilon}}+ yx^{3/2}T^{\epsilon}}.\numberthis\label{I1integral}
     \end{align*}
     The integral $I_2$ is given as \begin{align*}
         I_2 & =\frac{\rho_{\mathbb{K}} y^2}{ (2\pi i)^2}\int_{\beta_1-iT}^{\beta_1+iT}\int_{\beta_2-iT}^{\beta_2+iT} \frac{\zeta_{\mathbb{K}}(2-s_1)\zeta_{\mathbb{K}}(1-s_1+s_2)}{(2-s_1)\zeta_{\mathbb{K}}(2-s_1+s_2)\zeta_{\mathbb{K}}(s_1)}\frac{x^{s_1+s_2}}{y^{s_1}s_1s_2}ds_1ds_2.
     \end{align*}
     We shift  the integral over $s_1$-plane into a rectangular contour with vertices $\beta_1+iT$, $3/2+iT$, $3/2-iT$, and $\beta_1+iT$. The integrand has a simple at $s_1=s_2$ with residue \[-\rho_{\mathbb{K}}\frac{\zeta_{\mathbb{K}}(2-s_2)x^{2s_2}/y^{s_2}}{s_2^2(2-s_2)\zeta_{\mathbb{K}}(2)\zeta_{\mathbb{K}}(s_2)}.\] 
     Let $L_{2,1}$, $L_{2,3}$ be the integrals along the horizontal lines of the contour, and $L_{2,2}$ is the integral along the vertical line. Then, we have 
     \begin{align*}
         |L_{2,1}|,|L_{2,3}|&\ll \frac{xy^2}{T^{2-\epsilon}}\int_{-T}^{T}\left(\int_{\beta_1}^{3/2}\frac{x^{\sigma}}{y^{\sigma}}d\sigma\right)\frac{1}{1+|t|}dt\\
         &\ll\frac{x^{5/2}y^{1/2}}{T^{2-\epsilon}} +\frac{x^2y}{T^{2-\epsilon}},\numberthis\label{l21}
     \end{align*}
     and
     \begin{align*}
         |L_{2,2}|&\ll 
         x^{5/2}y^{1/2}\int_{-2T}^{2T}\frac{|\zeta_{\mathbb{K}}(-1/2+\beta_2+it)|}{|\zeta_{\mathbb{K}}(1/2+\beta_2+it)|}\int_{-T}^{T}\frac{t_1^{\epsilon}}{(1+|t_1|)^2(1+|t_1+t|)}dt_1dt\\
         &\ll\frac{x^{5/2}y^{1/2}}{T^{1-\epsilon}}.\numberthis\label{l22}
     \end{align*}
     Therefore, the integral $I_2$ becomes \begin{align*}
         I_2&=-\frac{\rho_{\mathbb{K}}^2 y^2}{\zeta_{\mathbb{K}}(2) (2\pi i)}\int_{\beta_2-iT}^{\beta_2+iT}\frac{\zeta_{\mathbb{K}}(2-s_2)x^{2s_2}/y^{s_2}}{s_2^2(2-s_2)\zeta_{\mathbb{K}}(s_2)}ds_2+\BigO{ \frac{x^{2}y}{T^{2-\epsilon}}+\frac{x^{5/2}y^{1/2}}{T^{1-\epsilon}}}\\
         &= \frac{\rho_{\mathbb{K}}^2\zeta_{\mathbb{K}}(0)}{4\zeta_{\mathbb{K}}(2)^2 }x^4+\BigO{ \frac{x^{2}y}{T^{2-\epsilon}}+\frac{x^{5/2}y^{1/2}}{T^{1-\epsilon}}}.\numberthis\label{I2integral}
     \end{align*}
     The integral $I_3$ is given as \begin{align*}
       I_3 & =\frac{\rho_{\mathbb{K}} y^2}{ (2\pi i)^2}\int_{\beta_1-iT}^{\beta_1+iT}\int_{\beta_2-iT}^{\beta_2+iT} \frac{\zeta_{\mathbb{K}}(2-s_2)\zeta_{\mathbb{K}}(1-s_2+s_1)}{(2-s_2)\zeta_{\mathbb{K}}(2-s_2+s_1)\zeta_{\mathbb{K}}(s_2)}\frac{x^{s_1+s_2}}{y^{s_2}s_1s_2}ds_1ds_2.  
     \end{align*}
   Next, we estimate $I_3$, we shift the integration over $s_2$ into the 
 rectangular contour with vertices $\beta_2+iT$, $\beta_2-iT$, $3/2+iT$, and $3/2-iT$, and denote the integration along the line $[\beta_2+iT,3/2+iT]$, $[3/2+iT,3/2-iT]$, and $[3/2-iT,\beta_2-iT]$ are $L_{3,1}$, $L_{3,2}$, and $L_{3,3}$, respectively. Therefore, we have
    \begin{align*}
        |L_{3,1}|,|L_{3,3}|& \ll \frac{xy^2}{T^2}\int_{-T}^{T}\left(\int_{\beta_2}^{3/2}T^{\epsilon}\frac{x^{\sigma}}{y^{\sigma}}d\sigma\right)\frac{1}{1+|t|}dt\\
         &\ll \frac{xy^2}{T^{2-\epsilon}}\int_{-T}^{T}\left(\int_{\beta_1}^{3/2}\left(\frac{x}{y}\right)^{\sigma}d\sigma\right)\frac{1}{1+|t|}dt\ll \frac{x^{5/2}y^{1/2}}{T^{2-\epsilon}}+\frac{x^2y}{T^{2-\epsilon}},\numberthis\label{l31}
    \end{align*}
   and
    \[|L_{3,2}|\ll\frac{x^{5/2}y^{1/2}}{T^{1-\epsilon}}.\numberthis\label{l32}\]
    Therefore , we have\[I_3=\BigO{ \frac{x^{5/2}y^{1/2}}{T^{1-\epsilon}}+\frac{x^2y}{T^{2-\epsilon}}}.\numberthis\label{I3integral}\]
    Finally, the integration $I_4$ is equal to \[I_4=\frac{\rho_{\mathbb{K}} }{ (2\pi i)^2}\int_{\beta_1-iT}^{\beta_1+iT}\int_{\beta_2-iT}^{\beta_2+iT} \frac{\zeta_{\mathbb{K}}(2-s_1)\zeta_{\mathbb{K}}(2-s_2)\zeta_{\mathbb{K}}(3-s_1-s_2)}{(3-s_1-s_2)\zeta_{\mathbb{K}}(4-s_1-s_2)\zeta_{\mathbb{K}}(s_1)\zeta_{\mathbb{K}}(s_2)}\frac{y^{3-s_1-s_2}x^{s_1+s_2}}{s_1s_2}ds_1ds_2.\]
    We estimate $I_4$ by shifting the integration over $s_2$ to the contour with vertices $\beta_2+iT$, $\beta_2-iT$, $5/2-\beta_1+iT$, and $5/2-\beta_1-iT$. The integration along the line $[\beta_2+iT,5/2-\beta_1+iT]$, $[5/2-\beta_1+iT,5/2-\beta_1-iT]$, and $[5/2-\beta_1-iT,\beta_2-iT]$ are  denoted by $L_{4,1}$, $L_{4,2}$, and $L_{4,3}$, respectively. Then 
    \begin{align*}
        |L_{4,1}|,|L_{4,3}|&\ll  \frac{x^{5/2}y^{1/2}}{T^{2-\epsilon}}+\frac{x^2y}{T^{2-\epsilon}}.\numberthis\label{l41}
    \end{align*}
    And \begin{align*}
        |L_{4,2}|&\ll\frac{x^{5/2}y^{1/2}}{T^{1-\epsilon}}.\numberthis\label{l42} 
    \end{align*}
    Thus, \[I_4=\BigO{ \frac{x^{5/2}y^{1/2}}{T^{1-\epsilon}}+\frac{x^2y}{T^{2-\epsilon}}}.\numberthis\label{I4integral}\]
Collecting the results from \eqref{I1integral}, \eqref{I2integral}, \eqref{I3integral}, \eqref{I4integral}, and insert in \eqref{integral}. We have
\begin{align*}
    I&=\frac{\rho_{\mathbb{K}}^2}{\zeta_{\mathbb{K}}(2)}yx^2+\frac{\rho_{\mathbb{K}}^2\zeta_{\mathbb{K}}(0)}{4\zeta_{\mathbb{K}}(2)^2 }x^4
   +\BigO{\frac{x^2y}{T^{1-\epsilon}}+ yx^{3/2}T^{\epsilon}+\frac{x^{5/2}y^{1/2}}{T^{1-\epsilon}}}.
   \end{align*}
  Choose $T=x^{2-\epsilon}$ and substitute in the above expression; we obtain the required result.
    \end{proof}

  \section{Proof of Theorem \ref{theorem3}}\label{sec6}
  Note that from \cite[Theorem 4.3]{washington}, for $\mathbb{K}=\mathbb{Q}(\zeta_m)$, we have
  \[\zeta_{\mathbb{K}}(s)=\zeta(s)\prod_{\chi\ne \chi_0}L(s,\chi).\]
  Here, $\chi$ is a non-principle Dirichlet character modulas $m$. Next, we use the bounds of zeta function and L-function \cite{Hu} and Theorem \ref{paralinde} to bound $\zeta_{\mathbb{K}}(\sigma+it)$: 
  \begin{displaymath}
 \zeta_{\mathbb{K}}(\sigma+it) \ll \left\{
   \begin{array}{lr}
     {t^{\frac{(3-4\sigma)\phi(m)}{6}}}, &\  0\le \sigma \le 1/2,\\
      {t^{\frac{(1-\sigma)\phi(m)}{3}}}, &\  1/2\le \sigma \le 1,\\
      ({ \log t})^{\phi(m)},  &\ 1\le\sigma \le 2,\\
       {1} , &\ \sigma \ge 2.
     \end{array}\numberthis\label{bound1}
    \right.
\end{displaymath}
We now proceed with a proof similar to that of Lemma \ref{lemmasix}, incorporating the upper bound on $\zeta_{\mathbb{K}}(\sigma+it)$ stated in \eqref{bound1}. This yields the following lemma:
\begin{lem}\label{lemmasix2}
Let  $\mathbb{K}=\mathbb{Q}(\zeta_m)$, $-1/6<\Re(z_1)=a_1<0$, $-1/6<\Re(z_2)=a_2<0$, and $-1/6\phi(m)<\Re(z_1+z_2)=a_1+a_2<0$, then we have \[\sum_{\substack{0<\mathcal{N}(\mathcal{I})\le x\\\mathcal{I}\subseteq \mathcal{O}_{\mathbb{K}}}}\sigma_{\mathbb{K},z_1}(\mathcal{I})\sigma_{\mathbb{K},z_2}(\mathcal{I})=\mathrel{R}_0+\BigO{x^{5/6+a_1/2-a_2\phi(m)/6}\log^{5\phi(m)}x}.\]
  \end{lem}
  Here, $R_0$ is the same as in Lemma \ref{lemmasix}. Following the proof of Theorem \ref{theorem2} using the above key Lemma \ref{lemmasix2} and bounds of Dedekind zeta function \eqref{bound1}, one obtains the required result.  As the underlying arguments remain unchanged, we omit a detailed exposition.
   \bibliographystyle{plain}
   \bibliography{ref}

\begin{thebibliography}{10}

\bibitem{balandraud2007application}
\'{E}. Balandraud.
\newblock An application of {R}amanujan sums to equirepartition modulo an odd
  integer.
\newblock {\em Unif. Distrib. Theory}, 2(2):1--17, 2007.

\bibitem{Carmi}
R.~D. Carmichael.
\newblock Expansions of {A}rithmetical {F}unctions in {I}nfinite {S}eries.
\newblock {\em Proc. London Math. Soc. (2)}, 34(1):1--26, 1932.

\bibitem{MR2869206}
T.~H. Chan and A.~V. Kumchev.
\newblock On sums of {R}amanujan sums.
\newblock {\em Acta Arith.}, 152(1):1--10, 2012.

\bibitem{CG}
S.~Chaubey and S.~Goel.
\newblock On the distribution of {R}amanujan sums over number fields.
\newblock {\em Ramanujan J.}, 61(3):813--837, 2023.

\bibitem{chaubey2023hardy}
S.~Chaubey, S.~Goel, and M~R. Murty.
\newblock On the hardy-littlewood prime tuples conjecture and higher
  convolutions of ramanujan sums.
\newblock {\em Functiones et Approximatio Commentarii Mathematici}, 1(1):1--15,
  2023.

\bibitem{HD}
H.~Delange.
\newblock On {R}amanujan expansions of certain arithmetical functions.
\newblock {\em Acta Arith.}, 31(3):259--270, 1976.

\bibitem{MR3332952}
Y.~Fujisawa.
\newblock On sums of generalized {R}amanujan sums.
\newblock {\em Indian J. Pure Appl. Math.}, 46(1):1--10, 2015.

\bibitem{AD}
A.~Hildebrand.
\newblock \"{U}ber die punktweise {K}onvergenz von {R}amanujan-{E}ntwicklungen
  zahlentheoretischer {F}unktionen.
\newblock {\em Acta Arith.}, 44(2):109--140, 1984.

\bibitem{Hu}
Guangwei Hu and Ke~Wang.
\newblock Higher moment of coefficients of {D}edekind zeta function.
\newblock {\em Front. Math. China}, 15(1):57--67, 2020.

\bibitem{Iwaniec}
H.~Iwaniec and E.~Kowalski.
\newblock {\em Analytic number theory}, volume~53 of {\em American Mathematical
  Society Colloquium Publications}.
\newblock American Mathematical Society, Providence, RI, 2004.

\bibitem{jutila2007distribution}
M.~Jutila.
\newblock Distribution of rational numbers in short intervals.
\newblock {\em Ramanujan J.}, 14(2):321--327, 2007.

\bibitem{konvalina1996generalization}
J.~Konvalina.
\newblock A generalization of {W}aring's formula.
\newblock {\em J. Combin. Theory Ser. A}, 75(2):281--294, 1996.

\bibitem{landau1949einfuhrung}
E.~Landau.
\newblock {\em Einf\"{u}hrung in die elementare und analytische {T}heorie der
  algebraischen {Z}ahlen und der {I}deale}.
\newblock Chelsea Publishing Company, New York, N. Y., 1949.

\bibitem{LR}
L.~G. Lucht and K.~Reifenrath.
\newblock Mean-value theorems in arithmetic semigroups.
\newblock {\em Acta Math. Hungar.}, 93(1-2):27--57, 2001.

\bibitem{ma2021average}
J.~Ma, H.~Sun, and W.~Zhai.
\newblock The average size of ramanujan sums over cubic number fields.
\newblock {\em arXiv preprint arXiv:2105.11699}, 2021.

\bibitem{nathanson1996additive}
M.~B. Nathanson.
\newblock {\em Additive number theory}, volume 165 of {\em Graduate Texts in
  Mathematics}.
\newblock Springer-Verlag, New York, 1996.
\newblock Inverse problems and the geometry of sumsets.

\bibitem{Nowak12}
W.~G. Nowak.
\newblock The average size of {R}amanujan sums over quadratic number fields.
\newblock {\em Arch. Math. (Basel)}, 99(5):433--442, 2012.

\bibitem{nowak13}
W.~G. Nowak.
\newblock On {R}amanujan sums over the {G}aussian integers.
\newblock {\em Math. Slovaca}, 63(4):725--732, 2013.

\bibitem{ramanujan1918certain}
S.~Ramanujan.
\newblock On certain trigonometrical sums and their applications in the theory
  of numbers [{T}rans. {C}ambridge {P}hilos. {S}oc. {\bf 22} (1918), no. 13,
  259--276].
\newblock In {\em Collected papers of {S}rinivasa {R}amanujan}, pages 179--199.
  AMS Chelsea Publ., Providence, RI, 2000.

\bibitem{ramare2007eigenvalues}
O.~Ramar\'{e}.
\newblock Eigenvalues in the large sieve inequality.
\newblock {\em Funct. Approx. Comment. Math.}, 37(part 2):399--427, 2007.

\bibitem{MR3600410}
N.~Robles and A.~Roy.
\newblock Moments of averages of generalized {R}amanujan sums.
\newblock {\em Monatsh. Math.}, 182(2):433--461, 2017.

\bibitem{schwarz1988ramanujan}
W.~Schwarz.
\newblock Ramanujan expansions of arithmetical functions.
\newblock {\em Ramanujan revisited (Urbana-Champaign, Ill., 1987)}, pages
  187--214, 1988.

\bibitem{SS}
W.~Schwarz and J.~Spilker.
\newblock {\em Arithmetical functions}, volume 184 of {\em London Mathematical
  Society Lecture Note Series}.
\newblock Cambridge University Press, Cambridge, 1994.
\newblock An introduction to elementary and analytic properties of arithmetic
  functions and to some of their almost-periodic properties.

\bibitem{shiu1980brun}
P.~Shiu.
\newblock A {B}run-{T}itchmarsh theorem for multiplicative functions.
\newblock {\em J. Reine Angew. Math.}, 313:161--170, 1980.

\bibitem{washington}
L.~C. Washington.
\newblock {\em Introduction to cyclotomic fields}, volume~83.
\newblock Springer Science \& Business Media, 1997.

\bibitem{AW}
A.~Wintner.
\newblock {\em Eratosthenian {A}verages}.
\newblock publisher unknown, Baltimore, Md., 1943.

\bibitem{EW}
E.~Wirsing.
\newblock Das asymptotische {V}erhalten von {S}ummen \"{u}ber multiplikative
  {F}unktionen.
\newblock {\em Math. Ann.}, 143:75--102, 1961.

\bibitem{zhai2021average}
W.~Zhai.
\newblock The average size of ramanujan sums over quadratic number fields.
\newblock {\em The Ramanujan Journal}, pages 1--17, 2021.

\end{thebibliography}
\end{document}